\documentclass{amsart}

\usepackage{amssymb, amsmath, latexsym, a4}
\usepackage[latin1]{inputenc}
\usepackage{enumerate}
\usepackage[all]{xy}

\newtheorem{De}{Definition}[section]
\newtheorem{Th}[De]{Theorem}
\newtheorem{Pro}[De]{Proposition}
\newtheorem{Le}[De]{Lemma}
\newtheorem{Co}[De]{Corollary}

\def\ss{\sigma}
\def\xto#1{\xrightarrow[]{#1}}

\def\xot#1{\xleftarrow[]{#1}}

\let\dd\delta
\def\id{{\sf Id}}

\newcommand\bI{{\boldsymbol I}}
\newcommand\bG{{\boldsymbol G}}
\newcommand\bF{{\boldsymbol F}}

\def\Aut{{\sf Aut}}

\def\Out{{\sf Out}}

\def\1{^{-1}}

\def\ob{{\bf Ob}}

\let\then\Rightarrow

\def\bb{\beta}
\def\ab{\sf Ab}

\newcommand{\grupo}[1]{\langle #1\rangle}

\begin{document}

\title{An outline of obstruction theories of extensions via track categories}
\author{MAriam Pirashvili}

\maketitle

\begin{abstract}
Abelian track categories can be classified via the third Baues-Wirsching cohomology of small categories. This approach is used in this paper to compare and classify different generalisations of the obstruction theory of non-abelian group extensions, due to Cegarra, Garz\'on and Grandjean, Cegarra, Garz\'on and Ortega, and Chen, Du and Wang.
\end{abstract}

\section{Introduction}
The $3$-dimensional cohomology groups appear in the study of extensions of non-abelian groups. The general problem is that of constructing all short exact sequences given end groups $G$ and $\Pi$ and given an abstract kernel $\eta: \Pi\to G$. In \cite{e-m}, Eilenberg and MacLane showed that the existence of such extensions depends on an element, referred to as the obstruction, in the third cohomology $H^3(\Pi,Z(G))$.

There have been numerous generalisations of this obstruction theory. In \cite{cegarra,ortega}, Cegarra et al. defined graded extensions of categories by a group, with and without a monoidal structure on the category. More recently, in \cite{chen}, the authors defined extensions of groupoids by another groupoid. These theories are also related by the fact that they are all special cases of Grothendieck cofibrations.

In fact, we show in this paper that using the Baues-Wirsching cohomology for small categories, we can give a unified viewpoint for all these theories. We construct a class in the third Baues-Wirsching cohomology, which in different settings connects all of these classes in a non-trivial way. We exploit the crucial observation of Baues-Jibladze that abelian track categories are classified by the third Baues-Wirsching cohomology. Any such track category defines a class in the third cohomology. We show that all the above-mentioned classes can be obtained by choosing the appropriate track category.

This allows for a unified view of these related construction, as well as allowing for comparisons between them.

In particular, to obtain all Eilenberg-Maclane classes, we construct an abelian track category which has not been considered before. The objects of said track category are groups with surjective homomorphisms as $1$-morphisms and $2$-morphisms given by conjugation. This gives a class in $H^3({\mathcal G}_\sim, D^{\mathcal G})$, where ${\mathcal G}_\sim$ is the homotopy category of said track category, and the associated natural system $D^{\mathcal G}$ is the functor which assigns to each group its centre. Then for each group, we can consider the restriction to the one-object subcategory, with $1$-morphisms the outer automorphisms of the group and the centre of the group as the associated natural system, and thus recover the associated Eilenberg-Maclane class.

\section{Third cohomology and abelian track categories}

Track categories are groupoid enriched categories. A track category is \emph{abelian} if the automorphism group of any 1-arrow is abelian. By a fundamental observation of Baues and Jibladze \cite{BJ}, any abelian track category $\mathcal T$ defines an element in third cohomology, called the \emph{global Toda class} of $\mathcal T$.

In this section we recall the basic definitions and facts needed to explain this result. 

\subsection{Category of factorizations and natural systems} For a category $\bI$, one denotes by $\bF\bI$ the
\emph{category of factorizations of} $\bI$ \cite{BW}. Let us recall
that objects of the category $\bF\bI$ are morphisms $\alpha:i\to j$
of $\bI$. A morphism from $\alpha$ to $\beta:k\to l$ in $\bF\bI$ is a
pair $(\xi,\eta)$, where $\xi:k\to i$ and $\eta:j\to l$ are
morphisms in $\bI$ such that
$$
\beta=\eta\circ\alpha\circ\xi.
$$
In other words, the following diagram
$$
\xymatrix{
j\ar[r]^\eta&l\\
i\ar[u]^\alpha&k\ar[l]^\xi\ar[u]_\beta }
$$
commutes. If $(\xi',\eta')$ is also a morphism in $\bF\bI$ from $\beta:k\to l$ to $\beta':k'\to l'$:

$$\xymatrix{
j\ar[r]^\eta&l\ar[r]^{\eta'}&l'\\
i\ar[u]^\alpha&k\ar[l]^\xi\ar[u]_\beta&k'\ar[l]^{\xi'}\ar[u]_{\beta'}
 }
$$
then the composite morphism $\alpha\to\beta'$ in $\bF\bI$ is defined by
$$
(\xi',\eta')(\xi,\eta) = (\xi\xi',\eta'\eta).
$$
Clearly $(\xi,\eta)$ is an isomorphism in $\bF\bI$ iff both $\xi$ and $\eta$ are isomorphisms in $\bI$.

Let $\bI$ be a small category. A \emph{natural system of  abelian groups} on $\bI$
is a functor $D$ from the category $\bF\bI$ to the category $\ab$ of abelian
groups. For a natural system $D$ we usually denote the value of $D$
on $\alpha:i\to j$ by $D_\alpha$ as well as $D(\alpha)$. If $\alpha$
is the identity $\id_i:i\to i$ we write $D_i$ instead of
$D_{\id_i}$. For morphisms $\xi:k\to i$ and $\eta:j\to l$ we also
denote actions of $D(\xi,\id_j):D_\alpha\to D_{\alpha\xi}$, resp.
$D(\id_i,\eta):D_\alpha\to D_{\eta\alpha}$ on an element $a\in
D_\alpha$ by $\xi^*(a)$, resp. $\eta_*(a)$ or as well by $a\xi$,
resp. $\eta a$.  It is clear that $$(\xi ,\eta
)=(\xi,\id_j)(\id_i,\eta ) =(\id_i,\eta)(\xi,\id_j),$$ i.~e. in
$D_{\eta\alpha\xi}$ one has $\xi^*\eta_*(a)=\eta_*\xi^*(a)$ for any $a\in D_\alpha$.
\subsection{Baues-Wirsching cohomology}

Let $D$ be a natural system of abelian  groups on a small category $\bI$.  Following Baues
and Wirsching \cite{BW} we define the cochain complex
$C^*(\bI;D)$ by
$$
C^0(\bI;D)=\prod_{i\in\ob(\bI)}D_i
$$
and
$$
C^n(\bI;D)=\prod_{i_0\xot{\i_1}\cdots\xot{\i_n}i_n}D_{\i_1\cdots
\i_n}, \quad n>0$$
The boundary operator is defined by $d=\sum_{i=0}^n (-1)^i\dd^i$, where the coface  operators $\dd^m:C^n(\bI;D)\to C^{n+1}(\bI;D)$ are
defined as follows. If $m=0$ one puts:
$$(\dd^0(f))(i_0\xot{\i_1}\cdots\xot{\i_{n+1}}i_{n+1})
   ={\i_1}_*f(i_1\xot{\i_2}\cdots\xot{\i_{n+1}}i_{n+1}),$$
For $0<m<n$ one puts
$$(\dd^m(f))(i_0\xot{\i_1}\cdots\xot{\i_{n+1}}i_{n+1})
   =f(i_0\xot{\i_1}\cdots\xot \i_{m-1}\xot{\i_m
\i_{m+1}}i_{\nu+1}\xot{}\cdots
    \xot{\i_{n+1}}i_{n+1})$$
and finally for $m=n$ one puts
$$(\dd^n(f))(i_0\xot{\i_1}\cdots\xot{\i_{n+1}}i_{n+1})
   =(\i_{n+1})^*f(i_0\xot{\i_1}\cdots\xot{\i_n}i_n)$$
for $n>0$. 

\begin{De} Let $D$ be a natural system of abelian groups on a small category $\bI$. The
 \emph{cohomology $H^*(\bI;D)$ of $\bI$ with coefficients in} $D$ is
  defined as the cohomotopy of the cochain complex $C^*(\bI;D)$.
\end{De}

%The cohomology of small categories has obvious functorial properties. For example, for fixed $\bI$, the groups $H^n(\bI,F)$ form a covariant functor on $F$ and any short exact sequence $0\to F_1\to F\to F_2\to 0$ gives rise to the long exact sequence of cohomology groups.

We will now discuss functorial properties with respect of the first variable. Let $q:{\bf C}\to \bI$ be a functor and $D$ be a natural system of abelian groups on $\bI$. Then we have a natural system $D_q$ on $\bf C$, given by $\alpha\mapsto D_{q(\alpha)}$, where $\alpha$ is a morphism of $\bf C$. In this notation there is a cochain map $q^*:C^*(\bI,D)\to C^*({\bf C},D_q)$ given by 
$$q^*(f)(c_0\xot{\alpha_1}\cdots\xot{\alpha_{n+1}}c_{n+1})=f(q(c_0)\xot{q(\alpha_1)}\cdots\xot{q(\alpha_{n+1})}q(c_{n+1}))$$
In particular we have induced map in cohomology
$$q^*:H^n(\bI,D)\to H^n({\bf C},D_q),\ n\geq0.$$
{\bf Of the special interests is the case, when $\bf C$ is a subcategory of $\bI$ and $q$ is the inclusion. In this case insteate $D_q$ we write simply $D$ and  the above map in cohomology is called the \emph{restriction homomorphism} and is denoted by $Res:H^*(\bI,D)\to H^*({\bf C},D)$.}

\subsection{Functors and bifunctors as natural systems}. There are functors
$$\bF\bI\xto{q} \bI^{op}\times \bI$$
and $$p_1: \bI^{op}\times \bI\to \bI^{op}, \quad p_2: \bI^{op}\times \bI\to \bI$$
given respectively by
$$q(\alpha:i\to j)=(i,j), \ p_1(i,j)=i, \ p_2(i,j)=j.$$
Thus any bifunctor $B: \bI^{op}\times \bI\to\ab$, resp. any covariant functor $F:\bI\to \ab$ or any contravariant functor $G:\bI^{op}\to \ab$, gives rise to natural systems on $\bI$ given by $D(\alpha:i\to j)=B(i,j)$ respectively $D(\alpha:i\to j)=F(j)$ or $D(\alpha:i\to j)=G(i)$. In what
follows we will consider the functors and bifunctors as natural systems in this way. Thus there are well-defined cochain complexes $C^*(\bI,B)$, $C^*(\bI,F)$, $C^*(\bI,G)$ and cohomologies $H^*(\bI,B)$, $H^*(\bI,F)$, $H^*(\bI,G)$.
The groups $H^*(\bI,F)$ are quite classical and coincide with right derived functors of the limit, studied for example in \cite{laudal}.

\subsection{The case of groupoids}  Let $\bG$ be  a small groupoid (for example, it can be a group, considered as a one object category). There is a functor
$\kappa:\bG\to \bF\bG$ which is given on objects by $\kappa(g)=(g\xto{\id_g} g)$. The functor $\kappa$ takes a morphism $\eta:g\to h$ of $\bG$ to the morphism $(\eta^{-1},\eta):\id_g\to \id_h$. In other words, one has a commutative diagram 
$$
\xymatrix{
g\ar[r]^\eta&h\\
g\ar[u]^{\id}&h\ar[l]^{\eta^{-1}}\ar[u]_\id }
$$

\begin{Le}\label{nsgpd} The functor $\kappa$ is an equivalence of categories.
\end{Le}

\begin{proof} Define the functor $\iota: \bF\bG\to \bG$ on objects by $$\iota(x\xto{\alpha} y)=y=codomain (\alpha)$$
and on morphisms by $\iota(\xi,\eta)=\eta$. Here $(\xi,\eta)$ is a morphism $\alpha\to \beta$ in the category $\bF\bG$, thus one has a commutative diagram
$$
\xymatrix{
y\ar[r]^\eta&v\\
x\ar[u]^\alpha&u\ar[l]^\xi\ar[u]_\beta }
$$ 

Then obviously $\iota\circ \kappa=\id_{\bG}$. Moreover there is a natural isomorphism of functors $\theta:\kappa\circ \iota\to \id_{\bF\bG}$. Here $\theta(x\xto{\alpha} y)=(\alpha,\id_y)$. This follows from the following commutative diagram
$$
\xymatrix{
y\ar[r]^\id&y\\
y\ar[u]^\id&x\ar[l]^\alpha\ar[u]_\alpha }
$$ 

\end{proof}

Let $\bG$ be  a small groupoid. Then there is an isomorphism of categories $\bG^{op} \to \bG$ which is the identity on objects and takes a morphism to its inverse. Hence the category of contravariant functors from $\bG$ to $\ab$ is isomorphic to the category of   covariant functors from $\bG$ to $\ab$ and both categories are equivelent to the category of natural systems on $\bG$ thanks to Lemma \ref{nsgpd}.

Especially nice is the case when $\bG$ is a one object category corresponding to a group $G$. In this case, a covariant functor $\bG\to\ab$ is nothing but a left $G$-module. Any such $A$ gives rise to a natural system $D$ on $\bG$, where $D_x=A$ for all $x\in G$. Moreover, for any $a\in A$, $x,y,z\in G$, the maps
$$y_*:D_{x}=A\to A=D_{xy}, \quad z^*:D_x=A\to A=D_{zx}$$
are given by $y_*(a)=ya$ and $z^*(a)=a$ respectively.
Conversely, if $D$ is a natural system on $\bG$ then one obtains a $G$ module $A$ as follows: $A=D_1$ (here $1$ is the unit of $G$ considered as the identity morphism of $\bG$). The left action of $G$ on $A$ is given by
$$xa=x_*{(x^{-1})}^*a$$
for $a\in A$ and $x\in G$. Thus, the category of natural systems on $\bG$ is equivalent to the category of left $G$-modules. By comparing the Baues-Wirsching cochain complex with the classical complex used in group cohomology, we see that $H^*(\bG,D)=H^*(G,A)$, where $A=D_1$ with the above actions. 

Actually this isomorphism can be generalised to groupoids. This follows from the fact that any groupoid is equivalent to a groupoid which is a disjoint union of one object groupoids (i.e. groups considered a one object categories). Denote by $\Lambda$ the set of connected components of $\bG$ and for each $\lambda\in \Lambda$ choose an object $x_\lambda$ in the connected component corresponding to $\lambda$. Denote by $G_\lambda$ the group of automorphisms of $x_\lambda$ then \begin{equation}\label{grcoh}H^*(\bG,D)\cong \prod_\lambda H^*(G_\lambda, A_\lambda)\end{equation}
where $A_\lambda$ is the evaluation of $D$ on the morphism ${\id}_{x_\lambda}:x_\lambda \to x_\lambda$.

\subsection{Track categories}  
Recall that a track category (known also as a groupoid enriched category) ${\bf T}$ has objects $A,B,C, \cdots $ and for any two objects $A$ and $B$ a small groupoid ${\bf T}(A,B)$ is given, called the hom-{\it groupoid} of ${\bf T}$.
Moreover, for any triple of objects $A,B,C$ we have a composition functor
$${\bf T}(B,C)\times {\bf T}(A,B) \to {\bf T}(A,C).$$
For any object $A$ the identity $1_A$ is given which is an object of ${\bf T}(A,A)$. These data must satisfy the usual equations of associativity and identity. Objects of the
category ${\bf T}(A,B)$ are called $1$-{\it morphisms} in ${\bf T}$, while morphisms
from ${\bf T}(A,B)$ are called $2$-{\it morphisms} or \emph{tracks}. $1$-morphisms are denoted by $f,g$
etc, while $2$-morphisms are denoted by $\alpha,\bb$ etc. If $f$ is a $1$-morphism
from $A$ to $B$ we write $f:A\to B$, while for a $2$-morphism $\alpha$ from
$f$ to $g$ we write $\alpha:f\then g$. For $1$-morphisms we use multiplicative notation, while for $2$-morphisms we use additive
notation. So we write $gf:A\to C$ for the composite  of $1$-morphisms  $f:A\to B$ and $g:B\to C$. We write $\bb+\alpha:f\then  h$ for
the composite of $2$-morphisms $\alpha:f\then  g$ and $\bb:g\then  h$.
Similarly, by $0_f$ or simply by $0$ we denote the identity morphism of the object $f$ in the category  ${\bf T}(A,B)$.

A $1$-morphism $g:B\to C$ induces the functors
$$g_*:{\bf T} (A,B)\to {\bf T}(A,C), \ \ f\mapsto gf, \ \ \alpha\mapsto g_*\alpha,$$
$$g^*:{\bf T} (C,D)\to {\bf T}(B,D), \ \ h\mapsto hg, \ \ \bb\mapsto g^*\bb.$$
These functors are restrictions of the composition functors. It follows from the definition that the following relations hold:
$$(\alpha +\bb)+\gamma=\alpha +(\bb+\gamma),\leqno {\rm TR \ 1}$$
$$\alpha +0=\alpha=0+\alpha,   \leqno {\rm TR \ 2}$$
$$f^*(\alpha +\bb)=f^*(\alpha) +f^*(\bb),\leqno {\rm TR \ 3} $$
$$g_*(\alpha +\bb)=g_*(\alpha) +g_*(\bb),\leqno {\rm TR \ 4}$$
$$f^*(0)=0=g_*(0),\leqno {\rm TR \ 5}$$
$$(ff_1)^*=f_1^*f^*, \ 1^*=1,\leqno {\rm TR \ 6}$$
$$(gg_1)_*=g_*g_{1*}, \ 1_*=1,\leqno {\rm TR \ 7}$$
$$g_*f^*=f^*g_*,\leqno {\rm TR \ 8}$$
$$f_1^*(\beta)+g_*(\alpha )=g_{1*}(\alpha)+f^*(\beta). \leqno {\rm TR \ 9}$$
The following diagram explains the 1-morphisms and $2$-morphisms in TR 9:
$$\xymatrix{A\ar@/^/[r]^{f} \ar@/_/[r]_{f_1}&B\ar@/^/[r]^{g}\ar@/_/[r]_{g_1} & C}, \ \ \alpha:f\then f_1, \ \ \beta:g \then g_1. $$
The equality TR 9 holds in ${\bf T}(gf,g_1f_1)$. The common value in TR 9 is denoted by $\beta * \alpha$ and is called the \emph{Godement product}.

A basic example of a track category is $\bf Cat$, the $2$-category of small categories. The objects of $\bf Cat$ are small categories, $1$-morphisms are functors and $2$-morphisms are natural isomorphisms. It has several interesting $2$-subcategories, for example the track category $\bf Gpd$ of groupoids, functors and their natural transformations.

Let us return to general track categories.

There are several categories associated to a track category ${\bf T}$.
The most important for us is the category ${\bf T} _0$, which has the same objects as ${\bf T}$, and the morphisms in ${\bf T}_0$ are the $1$-morphisms of ${\bf T}$. 

Another category which can be associated to ${\bf T}$ is the category ${\bf T}_1$. It has the same objects as ${\bf T}$. The morphisms $A\to B$ in ${\bf T}_1$ are triples $(f,f_1,\alpha)$ where $f,f_1:A\to B$ are $1$-morphisms in ${\bf T}$ and $\alpha:f\then f_1$ is a $2$-morphism in $\bf T$. The composition in ${\bf T}_1$ is defined by $$(g,g_1,\beta)\circ (f,f_1,\alpha)=(gf,g_1f_1,\beta*\alpha).$$
%where
%$$\alpha*\alpha_1=g^*(\alpha _1)+f_{1*}(\alpha )= g_{1*}(\alpha)+f^*(\alpha _1).$$
One then has the source and target functions
$${\bf T}_0\buildrel s\over \longleftarrow {\bf T}_1 \buildrel t\over
\longrightarrow  {\bf T}_0,$$ where both functors are identity on objects and on morphisms are given by $s(f,f_1,\alpha)=f$ and
$t(f,f_1,\alpha)=f_1$. Sometimes we also write
$${\bf T}_1 \rightrightarrows {\bf T}_0$$ instead of a track category ${\bf T}$.

The \emph{homotopy category} ${\bf T}_\sim$ of  a track category $\bf T$ is the category whose objects are the same as for $\bf T$, while morphisms are homotopy classes of $1$-arrows of $\bf T$. Recall that two  $1$-arrows $f,g:A\to B$ are \emph{homotopic} if there is a track $\alpha:f\then g$.

A \emph{strict $2$-functor} $F:{\bf T}\to{\bf T}'$ from a track category ${\bf T}$ to a
track category ${\bf T}'$  assigns to each $A\in\ob{\bf T}$ an object
$F(A)\in\ob({\bf T} ')$, to each 1-morphism $f:A\to B$ in ${\bf T}$ -- a 1-morphism
$F(f):F(A)\to F(B)$ in ${\bf T}'$, and to each 2-morphism $\alpha:f\then g$ for
$f,g:A\to B$, a 2-morphism $F(\alpha):F(f)\then F(g)$ in a functorial
way, i.~e. so that one gets functors
$$F_{A,B}:{\bf T}(A,B)\to {\bf T} '\left(F(A),F(B)\right).$$ Moreover
these assignments are compatible with identities and composition, or
equivalently induce a functor ${\bf T}_1\to{{\bf T}'}_1$, that is,
$F(1_A)=1_{F(A)}$ for $A\in\ob({\bf T})$, $F(fg)=F(f)F(g)$,
$F(\alpha*\beta)=F(\alpha)*F(\beta)$ and $F(0)=0$.

We also need a weaker version which is called a \emph{$2$-functor}. We will need it only in the case when the source category is an ordinary category, considered as a track category (where the only $2$-morphisms are $0:f\then f$).

\begin{De} \label{2fun} Let ${\bf T}$ be a $2$-category and let $\bI$ be a category. A $2$-functor $F:\bI\dashrightarrow {\bf T}$ consists of the following data:

\smallskip
-- an object $F(i)$ for each object $i\in\bI$,

\smallskip
-- a $1$-morphism $F(\alpha):F(i)\to F(j)$ for each morphism $\alpha:i\to j$ in $\bI$,

\smallskip
-- a track $F(\alpha,\beta):F(\alpha) F(\beta)\then F(\alpha\beta)$ for all composable arrows $$\xymatrix{i\ar[r]^{\beta}&j\ar[r]^{\alpha}&k}$$ of the category $\bI$.

\smallskip
\noindent One requires that the following conditions hold:

{\rm (i)} $F(\alpha,\id)=0$ and $F(\id,\alpha)=0$,

{\rm (ii)} for all composable arrows
$$\xymatrix{i\ar[r]^{\gamma}&j\ar[r]^{\beta}&k\ar[r]^{\alpha}&l}$$
of the category $\bI$ one has the following equality in ${\bf T}(F(\alpha\beta\gamma), F(\alpha) F(\beta) F(\gamma))$:

\begin{equation}
\label{pseudofunctor}
F(\alpha,\beta\gamma)+F(\alpha)_*f(\beta,\gamma)=F(\alpha\beta,\gamma)+F(\gamma)^*F(\alpha,\beta).
%\tag{*}
\end{equation}
 %% $$A(\gamma)_*A(\bb,\al)+A(\gamma,\bb\al)=A(\al)^*A(\gamma,\bb)+A(\gamma\bb,\al).$$
\end{De}

\subsection{Track categories and natural systems} 
Recall that ${\bf T}_0$ denotes the underlying category of a track category ${\bf T}$. For any morphism $f:A\to B$ of ${\bf T}_0$ we let $\Aut(f)$ be the collection of all automorphisms of $f$ in the category ${\bf T}(A,B)$. Thus, this is the collection of all $2$-morphisms $\alpha:f\then f$. It follows from TR 1  and TR 2 that $\Aut(f)$ is a group. Moreover, for any morphism $g:B\to C$ of the category ${\bf T}_0$, we have maps $g_*:\Aut(f)\to \Aut(gf)$ and $f^*:\Aut(g)\to  \Aut(gf)$, which are group homomorphisms thanks to TR 3 -- TR 5. Moreover, in this way one obtains a natural system $\Aut$ of groups on ${\bf T}_0$. This follows from the identities TR 6  --  TR 8. The natural system $\Aut$ is sometimes denoted by $\Aut^{\bf T}$ in order to explicitly show the dependence on the track category $\bf T$.

\subsection{Abelian track categories, linear track extensions and the global Toda class} Recall that a groupoid $\bG$ is \emph{abelian} if for any object $x$ of $\bG$ the group of authomorphisms of $x$ is abelian. 

A track category ${\bf T}$ is an {\it abelian track category} if $\bf T(A,B)$ is an abelian groupoid for all $A,B\in\ob \bf T$. In this case $\Aut^{\bf T}$ is a natural system of abelian groups.

\begin{De} \cite{BJ} Let $D$ be a natural system of abelian groups on a category ${\bf C}$. A \emph{linear track extension} of $\bf C$ by $D$, denoted by
$$D + \xto{\sigma} {\bf T}_1 \rightrightarrows {\bf T}_0\xto{q} {\bf C},$$
is a track category ${\bf T}$ equipped with a functor $q:{\bf T}_0\to {\bf C}$ such that $q$ induces an isomorphism of categories ${\bf T}_\sim\to {\bf C}$, and an isomorphism $\sigma: D_q\to \Aut^{\bf T}$, where $D_q$ is a natural system of abelian groups on ${\bf T}_0$ given by $f\mapsto D_{q(f)}$. Thus, such a linear extension consists of a collection of isomorphisms of groups $\sigma_f:D_{q(f)}\to Aut(f)$ for each $1$-arrow $f:A\to B$ of ${\bf T}$, which have the following properties.
\begin{enumerate}[(i)]
\item The functor $q$ is full and the identity on objects. In addition, for $f,g:A\to B$ in ${\bf T}_0$ we have $q(f)=q(g)$ iff there exists a track $f \then g$. In other words, the functor $q$ identifies $\bf C$ with ${\bf T}_\sim$.

\item For $\alpha:f\then g$ and $\xi:D_{q(f)}=D_{q(g)}$ we have $\sigma_f(\xi)=-\alpha +\sigma_g(\xi)+\alpha$.

\item For any 1-arrows $A\xto{f} B\xto{g} C\xto{h} D$ in ${\bf T}_0$ and any $\xi\in D_{q(g)}$ one has
$$h_*\sigma_g(\xi)=\sigma_{hg}(q(h)_*\xi), \quad f^*\sigma_g(\xi)=\sigma_{gf}( q(f)^*\xi).$$
\end{enumerate}
\end{De}

Given a $2$-functor $F: {\bf B}\to {\bf I}$, we can compose with $q$ to get a functor ${\bar F}: {\bf I}\to {\bf C}$. This is called a \emph{collective character}, following the classic terminology by Teichm\"uller \cite{teich}. Conversely, we can ask whether for a given collective character ${\bar F}: {\bf I}\to {\bf C}$ there exists a lifting as a 2-functor $F:{\bf I}\to {\bf T}$ and if it exists, how to classify all such liftings.
As we will see, the answer depends on an obstruction class in third cohomology. This construction is analogous to the one by Teichm\"uller for a similar situation with linear algebras, and to that by Eilenberg-MacLane, for the study of obstructions to group extensions with non-abelian kernels. 

To state the corresponding results, we will fix some notation. If $F: {\bf I}\to {\bf T}$ is a 2-functor and ${\bar F}: {\bf I}\to {\bf C}$ is the corresponding collective character, then for any object $i$ of ${\bf I}$ we have ${\bar F}(i)=F(i)$. Thus if a given collective character ${\bar F}: {\bf I}\to {\bf C}$ has a lifting $F$, then $F(i)$ must be ${\bar F}(i)$. Thus we are searching to find a pair $(F_1,F_2)$ of functions known as a factor set or $2$-cocycle. The function $F_1$ assigns a $1$-morphism $F_1(\alpha)$ to each morphism $\alpha:i\to j$ in $\bf I$ with the property $qF_1(\alpha)={\bar F}(\alpha)$, and a track $F_2(\alpha,\beta):F(\alpha)F(\beta)\then F(\alpha\beta)$ for all composable arrows $i\xto{\beta}j\xto{\alpha} k$ such that the conditions (i) and (ii) of Definition \ref{2fun} hold. 
Two such $2$-cocycles $(F_2,F_1)$ and $(F_2',F_1')$ are equivalent if there exists a function $\varphi$, which assigns to any arrow $\alpha:i\to j$ a track $\varphi(\alpha):F_1(\alpha)\then F_1'(\alpha)$ such that for all composable arrows
$$\xymatrix{i\ar[r]^{\beta}&j\ar[r]^{\alpha}&k}$$ of the category $\bI$ one has the following commutative diagram of tracks
$$\xymatrix{F_1(\alpha)\circ F_1(\beta)\ar[rr]^{F_2(\alpha,\beta)}\ar[dr]_{F_1(\alpha)_*\varphi(\beta)}&&F_1(\alpha\beta)\ar[r]^{\varphi(\alpha\beta)}&F_1'(\alpha\beta)\\
&F_1(\alpha)\circ F_1'(\beta)\ar[rr]_{F_1'(\beta)^*\varphi(\alpha)}&&F_1'(\alpha)\circ F_1'(\beta)\ar[u]^{F_2'(\alpha,\beta)}
}.$$
In other words, one has the following equality in ${\bf T}(F_1(\alpha)\circ F_1(\beta), F_1'(\alpha)\circ F_1'(\beta))$
%$$F_2'(\alpha,\beta) = -\varphi(\alpha)F_1'(\beta)^*-F_1(\alpha)_*\varphi(\beta)+F_2(\alpha,\beta)+\varphi(\alpha\beta).$$
$$F_2'(\alpha,\beta) = \varphi(\alpha\beta)+F_2(\alpha,\beta)-F(\alpha)_*\varphi(\beta)-F_1'(\beta)^*\varphi(\alpha).$$

The set of all such factor sets up to equivalence with fixed ${\bar F}$ is denoted by $H^2({\bf I}, {\bf T})$.

Let $D + \xto{\sigma} {\bf T}_1 \rightrightarrows {\bf T}_0\xto{q} {\bf C}$ be a linear track extension of a category $\bf C$ by a natural system $D$. Then \cite{BJ} the \emph{universal Toda bracket} $\grupo{\bf T}$ is the element of $H^3({\bf C},D)$ represented by the following cocycle: choose for each morphism $ f$ of ${\bf T}_\sim$ a representative 1-arrow $s(f)$, that is $q(s(f))=f$. Furthermore, choose a track $\varphi(f,g):s(f)s(g)\then s(fg)$ for any composable pair $A\xto{g}B\xto{f} C$ of the category $\bf C$. Such pair $(s, \varphi)$ is called a \emph{section} of $q$. 
It follows that $h^*\varphi_{f,g}+\varphi_{fg,h}$ and $f_*\varphi_{g,h}+\varphi_{f,gh}$ both define tracks $s(fgh)\then s(f)s(g)s(h)$ and hence
\begin{equation}\label{deftd}
%c(f,g,h)=f_*\varphi(g,h)+ \varphi(f,gh)-\varphi(fg,h)-h^*\varphi(f,g)
f_*\varphi(g,h)+ \varphi(f,gh)-\varphi(fg,h)-h^*\varphi(f,g)
\end{equation}
determines an element in $\Aut(s(f)s(g)s(h))$. Going back via $\sigma$, this determines an element $c(f,g,h)\in D_{fgh}$. Varying $f,g,h$ one obtains a 3-dimensional cocycle, whose class in cohomology is $\grupo{{\bf T}}\in H^3({\bf C},D)$.

\begin{Le}\label{1.5.m}
\label{cohomology}
\begin{enumerate}[(i)]
\item Let $q:{\bf T}\to {\bf C}$ be the canonical strict $2$-functor, which assigns the identity to all tracks (here and elsewhere, categories are considered as track categories with trivial tracks). Then $q$ has a section as a 2-functor iff $\grupo{\bf T}=0$.
\item More generally, given a category ${\bf I}$ and a functor $\bar F:{\bf I}\to {\bf C}$, we can lift the functor $\bar F$ to ${\bf T}$ as a $2$-functor if and only if the pull-back $Obs_{\bar F} = {\bar F}^*(\grupo{{\bf I}})\in H^3({\bf I}, D_{\bar F} )$ is zero.
\item Suppose we are given a collective character $\bar F:{\bf I}\to {\bf C}$ which has  a lift as a $2$-functor $F: {\bf I}\to {\bf T}$, so the obstruction $Obs_{\bar F}$ is zero. Then the set of all such liftings of $\bar F$ up to equivalence is in one-to-one correspondence with elements in $H^2({\bf I}, D_{\bar F})$. In fact, the set of such linear track extensions up to equivalence is a torsor over $H^2({\bf I}, D_{\bar F})$.
\end{enumerate}
\end{Le}
\begin{proof}
\begin{enumerate}[(i)]
\item Let $(s,\varphi)$ be a section of $q$. For it to be a $2$-functor, the equation (\ref{pseudofunctor}) has to hold. This is equivalent to the vanishing of the $3$-cocycle $c$ defined in the equation (\ref{deftd}). Conversely, let $c$ defined in the equation (\ref{deftd}) be a coboundary, i.e. $c(f,g,h)=f_*\rho(g,h)+ \rho(f,gh)-\rho(fg,h)-h^*\rho(f,g)$ for a function $\rho$ which assigns to each composable pair of morphisms $A\xto{g}B\xto{f} C$ an element in $D_{fg}$. We have a corresponding track $\bar \rho (f,g): = \sigma_{s(f)s(g)}(\rho(f,g))\in \Aut(s(f)s(g))$. 
%So $\sigma(c)$ determines an element in $Aut(s(f)s(g)s(h))$. 
We define $\tau = \varphi - \bar\rho.$ It is clear that $\tau(f,g)$ is a track $s(f)s(g)s(h)\then s(fgh)$. So, the pair $(s,\tau)$ is also a section, for which the corresponding $3$-cocycle $c_\tau(f,g,h)$ vanishes and hence $(s,\tau)$ defines a section of $q$ which is a $2$-functor.
\item This is a formal consequence of part (i) and the pull-back construction.
\item %The isomorphism classes of linear track extensions correspond to equivalence classes of collective characters. 
We describe an action
$$H^2({\bf I}, D_{\bar F})\times H^2({\bf I}, {\bf T})\to H^2({\bf I}, {\bf T}).$$
Let $h$ be a normalised $2$-cocycle representing the class of $[h]\in H^2({\bf I}, D_{\bar F})$, and let $(F_2,F_1)$ be a factor set representing the class $[F_2,F_1]\in  H^2({\bf I}, {\bf T})$. Then the result of the action is the class represented by the factor set $(F_2',F_1)$, where 
$$F_2'(\alpha,\beta)= \sigma_{F_1(\alpha\beta)}(h(\alpha,\beta)) +F(\alpha,\beta).$$
%for each composable pair of arrows in $I$, we consider $(h+f)(\alpha,\beta):F(\alpha)F(\beta)\to F(\alpha\beta)$, the track defined by
%$$(h+f)(\alpha,\beta) = h(\alpha,\beta)+f(\alpha,\beta).$$
To see that this is well-defined, and the action is transitive follows closely to the proof of Theorem $9$ in \cite{cegarra}.
\end{enumerate}
\end{proof}

The most important example of linear track extensions arises from abelian track categories \cite{BJ}. Namely, if $\bf T$ is an abelian track category, then obviously $\Aut^{\bf T}$ is a natural system of abelian groups on ${\bf T}_0$. It was proved by Baues and Jiblazde that there exists a unique (up to isomorphism) natural system of abelian groups $D^{\bf T}$ on ${\bf T}_\sim$ and an isomorphism of natural systems $\sigma:D^{\bf T}_q\to \Aut^{\bf T}$. Hence any such track category defines an element $\grupo{{\bf T}}\in H^3({\bf T}_\sim, D^{\bf T})$.

\section{Applications to Crossed modules, the Eilenberg-MacLane obstruction class and the Teichm\"uller class}

\subsection{Crossed modules and track categories}
Recall that a crossed module is a group homomorphism  $\delta : T\to R$ together with an action of $R$ on $T$ satisfying: 
$$\delta (^rt)=r\delta (t)r^{-1} \ {\rm and} \
 ^{\delta t}s=tst^{-1}, \ r\in R, t,s\in T.$$
It follows from the definition that the image $Im (\delta)$ is a normal subgroup of $R$, and the kernel $Ker (\delta)$ is in the centre of $T$. Moreover, the action of $R$ on $T$ induces an 
action of $G$ on $M=Ker (\delta)$, where $G=Coker(\delta)$. So we have an exact sequence 
$$0\to M\to T\xto{\delta} R\xto{p} G\to 0,$$
called a \emph{crossed extension} of a group $G$ by a $G$-module $M$. Such an extension defines an element in $H^3(G,M)$, see for example \cite[Ch. IV. Section 5] {kenn}. We recall the construction of this class.

Choose a pair of maps $(s:G\to R$, $\ss:G\times G\to T)$ for which the following hold:
$$ps(x)=x, \quad s(x)s(y)=\delta(\ss(x,y))s(xy), \ x,y\in G.$$
Then the map 
\begin{equation}\label{3cycle} 
f(x,y,z)=\,^x \sigma(y,z)\ss(x,yz)\ss(xy,z)^{-1}\ss(x,y)^{-1}\end{equation}
defines a $3$-cocycle $f\in Z^3(G,M)$. The class of this cocycle in $H^3(G,M)$ is denoted by $\grupo{\delta}$.

We will show that this classical construction is a special case of the global Toda bracket construction. To this end, for a group $L$ we let $\underline{L}$ denote the one object category whose morphisms are the elements of $L$. Denote by ${\bf T}^\delta$ the following track category. It has only one object, 1-arrows are elements of $R$ and the composition is induced by the product rule in $R$. A track from $r$ to $r'$ in ${\bf T}^{\delta}$ is an element $t\in T$ such that $r'=\delta(t)r$ (observe that now we use multiplicative notation for tracks, which should not cause any complications) and we write $t:r\then r'$. Moreover, for a track $t:r\then r'$ and $x,y\in R$, considered as 1-arrows, we set $$x_*(t)=\, ^xt:xr\then xr' \quad {\rm and} \quad y^*(t)=t:ry\then r'y.$$
One easily sees that we indeed obtain a track category, with
$$({\bf T}^\delta)_0=\underline{R}, \quad ({\bf T}^\delta)_\sim=\underline{G}.$$
Furthermore, the natural system $D^{\bf T}$ is nothing but the $G$-module $M$. This easily follows from the fact that $\Aut(r)=M$ for all $r\in R$.

\begin{Le}
\label{Toda}
One has the equality
$$\grupo{{\bf T}^{\delta}}=\grupo{\delta}.$$
\end{Le}
\begin{proof} By the isomorphism (\ref{grcoh}) the cohomology  groups in questions are the same. The rest follows by comparing the equations (\ref{deftd}) and (\ref{3cycle}) and the fact that $y^*(t)=t$ for all $y\in R$, $t\in T$.
\end{proof}

\subsection{The track category related to the category of groups and the Eilenberg-MacLane obstruction class}

Recall that to given groups $\Pi$, $G$ and a group homomorphism $\eta:\Pi\to \Out(G)$, Eilenberg and MacLane associated an obstruction class ${\mathcal Obs}_{\eta}\in H^3(\Pi,{\sf Z}(G))$, where ${\sf Z}(G)$ denotes the centre of $G$. In this section we construct a class, which is independent of the groups $\Pi$ and $G$, and which restricts to all these classes, see Proposition \ref{u_cl} below. Before we state this result explicitly, let us recall some basic facts related to the class ${\mathcal Obs}_{\eta}$, following \cite{homology}.

If
$$E: \quad \quad 1\to G\xrightarrow{\varkappa} B \xrightarrow{\sigma} \Pi \to 1$$
is a group extension, the epimorphism $\sigma$ induces a homomorphism $\eta: \Pi \to \Out(G)$, where $\Out(G)$ is the group of outer automorphisms of $G$. This happens in the following way: We have an action of $B$ on $G$ via conjugation, which gives us the homomorphism $\theta: B\to \Aut(G)$. Because $\theta(\varkappa (G))\subset {\sf In}(G)$, where ${\sf In}(G)$ is the group of inner automorphisms, we have the induced homomorphism $\eta: \Pi \to \Aut(G) / {\sf In}(G) = \Out(G)$.

We say that the extension $E$ has \emph{conjugation class $\eta$}: thus $\eta$ records in which way $G$ appears as a normal subgroup of $B$. Conversely, call a pair of groups $\Pi$, $G$ together with a homomorphism $\eta: \Pi\to \Out(G)$ an \emph{abstract kernel}.

The general problem of group extensions is constructing all extensions $E$ to a given abstract kernel $(\Pi, G, \eta)$. Given such an abstract kernel however, there might not necessarily exist an extension of $\Pi$ by $G$ that induces $\eta$.

The obstruction to the existence of an extension $$1\to G\to B \to \Pi \to 1$$ that induces $\eta$ is given by a certain class ${\mathcal Obs}_\eta$ in the third cohomology group $H^3(\Pi, {\sf Z}(G))$, where ${\sf Z}(G)$ is the centre of $G$.

To each group $G$ we associate the canonical crossed module $\mu: G\to \Aut(G)$. The kernel and cokernel of $\mu$ make up the crossed extension
$$0\to {\sf Z}(G)\to G\to \Aut(G)\to \Out(G)\to 1.$$
Here, ${\sf Z}(G)$ is the kernel of $\mu$, while $\Out(G)$ is the cokernel. As outlined in the previous section, this crossed extension leads to an element $Cl_G$ in the third cohomology group $H^3(\Out(G), {\sf Z}(G))$. Now we  use the pull-back construction of the homomorpism $\eta$ to create another crossed extension of $\Pi$ via ${\sf Z}(G)$:
$$
\xymatrix{
0\ar[r]&{\sf Z}(G)\ar[r]\ar[d]_{id}&G\ar[r]\ar[d]^{id}&X\ar[r]\ar[d]&\Pi\ar[d]^\eta\ar[r]&1\\
0\ar[r]&{\sf Z}(G)\ar[r]&G\ar[r]&\Aut(G)\ar[r]&\Out(G)\ar[r]&1.}
$$
The class of this crossed extension in $H^3(\Pi,{\sf Z}(G))$ is denoted by ${\mathcal Obs}_{\eta}$. Thus we have the basic formula
\begin{equation}\label{obs=cl} {\mathcal Obs}_{\eta}=\eta^*(Cl_G),\end{equation}
where $\eta^*$ is the induced homomorphism $H^3(\Out(G),{\sf Z}(G))\to H^3(\Pi,{\sf Z}(G))$. 
%
%
%Then we have the following result.
%
%\begin{Pro}
%For every abstract kernel $(\Pi,G,\eta)$ there is an induced %homomorphism $\eta^*: H^3(\Out(G),{\sf Z}(G))\to H^3(\Pi,{\sf %Z}(G))$ which maps $Cl_G$ to the obstruction element defined %above. We have ${\mathcal Obs}_{\eta}=\eta^*(Cl_G)$.
%\end{Pro}
%
%To see this, 
%The induced homomorphism $\eta^*: H^3(\Out(G),{\sf Z}(G))\to %H^3(\Pi,{\sf Z}(G))$ gives us the equality. This result is a %paraphrasing of the theorem in \cite{homology}:

Now we are in the position to formulate the following classical result of Eilenberg and MacLane, see \cite[Ch.IV. Section 6]{kenn} or \cite[Ch. IV] {homology}.
\begin{Th}(Eilenberg-MacLane)

The abstract kernel $(\Pi,G,\eta)$ has an extension if and only if ${\mathcal Obs}_{\eta}=0$.
%one of its obstructions is the cochain identically $0$.
\end{Th}

Next, we construct a universal class $\grupo{\mathcal G}$, such that all obstruction classes ${\mathcal Obs}_{\eta}$ are restrictions of $\grupo{\mathcal G}$. To this end, let us consider the following track category $\mathcal G$. The objects of $\mathcal G$ are groups, the $1$-morphisms are surjective homomorphisms, and the set of $2$-morphisms between two $1$-morphisms $f,g:G\to H$ is given by $\{b\in H | f(a) = b+g(a)-b\}\subseteq H$.

Let us denote the underlying category by ${\mathcal G}_0$, and the homotopy category by ${\mathcal G}_\sim$. Thus ${\mathcal G}_0$ is the category whose objects are groups and whose morphisms are surjective group homomorphisms, while the category ${\mathcal G}_\sim$ has as objects groups, while the morphisms are conjugacy classes of surjective homomorphisms. By our description of 2-morphisms, for each morphism $f:G\to H$ the group $\Aut(f)$ is the group $${\sf Z}_f = \{b\in H | f(a)+b=b+f(a) \}.$$ Since $f$ is surjective, we see that ${\sf Z}_f$  is the centre of $H$ and hence is abelian. Therefore, we have an abelian track extension 
$$0\to D^{\mathcal G}\to {\mathcal G}_1 \rightrightarrows {\mathcal G}_0\to {\mathcal G}_\sim\to 1,$$
with the natural system $D$ being given by abelian groups
$$D^{\mathcal G}_{f:G\to H}={\sf Z}(H).$$
It follows that $D^{\mathcal G}$ in our case is a natural system induced by the functor $${\mathcal G}_\sim\to \ab, \quad H\mapsto {\sf Z}(H).$$
Thus we obtain the class $\grupo{{\mathcal G}}\in H^3({\mathcal G}_\sim, D^{\mathcal G})$, where the last group is the cohomology of the category ${\mathcal G}_\sim$ with coefficients in a functor $H\mapsto {\sf Z}(H)$.  

{\bf Remark}. Since the track category $\mathcal G$ is not small, we have to be more careful in this place. To avoid set theoretical problems, we have to fix two universes ${\mathfrak U}_1\subset {\mathfrak U}_2$. The elements of ${\mathfrak U}_1$ are called sets, while elements of ${\mathfrak U}_2$ are called classes. The  objects of our categories are classes, while 
morphisms between two fixed objects form a set.  In this framework the definition of a cochain complex of a category still makes sense, but as a cochain complex of classes, and hence, cohomologies are not sets in general, but classes. Regarding our primary interest $H^*({\mathcal G}_\sim, D^{\mathcal G})$, we conjecture that nevertheless they are sets. The same reasoning is applied to other non-small categories considered below.

Let $\eta:G\to \Out(\Pi)$ be an abstract kernel. Then $\eta$ can be considered as a functor $\underline{\eta}:\underline{G}\to {\mathcal G}_\sim$, which sends the unique object of $\underline{G}$ to $\Pi\in {\mathcal G}_\sim$. 

\begin{Pro}\label{u_cl}
For every abstract kernel  $\eta:G\to \Pi$, we have the induced homomorphism 
$$\underline{\eta}^*:H^3({\mathcal G}_\sim, D^{\mathcal G})\to H^3(\Out(G),{\sf Z}(G))$$
and the following equality 
\begin{equation}\label{obs=eta} {\mathcal Obs}_{\eta}=\underline{\eta}^*(
\grupo{\mathcal G}).\end{equation}
Moreover, if we identify the one object category $\underline{\Out(\Pi)}$ with a subcategory of ${\mathcal G}_\sim$ which has a single object $\Pi$ and morphisms are isomorphisms in ${\mathcal G}_\sim$, then we have 
\begin{equation}\label{cl=res} Cl_\Pi=Res (\grupo{{\mathcal G}})). \end{equation}
\end{Pro}

\begin{proof} Thanks to the equality (\ref{obs=cl}), the  equality (\ref{cl=res}) implies (\ref{obs=eta}). To show the equality (\ref{cl=res}), we consider the crossed module $\delta:\Pi\to \Aut(\Pi)$, where $\delta$ sends element of $\Pi$ to the inner automorphisms of $\Pi$. The rest follows from the fact that the track category ${\bf T}^\delta$ (see the previous section) is isomorphic to the one object track subcategory of $\mathcal G$  with object the group  $\Pi$, where $1$-morphisms are isomorphism of $\Pi$, while 2-morphisms are the same as in ${\mathcal G}$.
\end{proof}

%Let $G$ be a group. Let us consider the one object sub-$2$-category of $\mathcal{T}_0$, with morphisms the automorphisms of $G$. The natural system associated to this track extension is ${\sf Z}(G)$, the centre of $G$. By Lemma \ref{Toda}, the thus obtained track category can be equated with the one obtained from the crossed module, leading to the same element $Cl_G$ in third cohomology. We have the following:

\subsection{The track category related to the category of rings and the Teichm\"uller class}

Let $A$ be a ring. Denote by $\Aut(A)$ the group of ring automorphisms of $A$. Denote by ${\sf U}(A)$ the group of invertible elements of $A$. There is an obvious homomorphism of groups $\partial:{\sf U}(A)\to \Aut(A)$, which sends an invertible element $a$ to the inner automorphism $\partial_a:A\to A$, where $\partial_a(x)=axa^{-1}$. Then ${\sf U}(A)\xto{\partial} \Aut(A)$ is a crossed module and 
$$e_A:\quad  0\to {\sf U}({\sf Z}(A))\to {\sf U}(A)\xto{\partial} \Aut(A)\to  \Out(A)\to 1$$
is a crossed extension, where ${\sf Z}(A)$ is the centre of $A$ and $\Out(A)={\sf Coker}(\partial)$.
This crossed module and the corresponding element $h_{A}$ in $H^3(\Out(A),{\sf U}({\sf Z}(A)))$ play an important role in the work of Huebschmann \cite{hueb} on the Teichm\"uller class in the Galois theory of rings. 

If $A$ is varied, one obtains different classes and the relationship between these classes in unclear. We will now prove that these classes are in fact a restriction of a unique class, which is independent of the chosen ring $A$.

Let us consider the following track category $\mathcal R$. The objects of $\mathcal R$ are rings, the $1$-morphisms are surjective homomorphisms of rings, and the set of $2$-morphisms between two $1$-morphisms $f,g:A\to B$ is given by 
$$\{b\in {\sf U}(B) | f(a) = bg(a)b^{-1}\}\subseteq B.$$
Therefore, for each morphism $f:A\to B$, $\Aut(f)$ is the abelian group ${\sf Z}_f = \{b\in {\sf U}(B) | f(a)b=bf(a) \}$, which is the centre of ${\sf U}(B)$. Let us denote the underlying category by ${\mathcal R}_0$, and the homotopy category by ${\mathcal R}_\sim$. Hence we have an abelian track extension 
$$0\to D^{\mathcal R}\to {\mathcal R}_1 \rightrightarrows {\mathcal R}_0\to {\mathcal R}_\sim\to 1,$$
where the natural system $D^{\mathcal R}$ assigns the abelian group ${\sf Z}(U(B))$ to a morphism $A\to B$ in ${\mathcal R}_\sim.$ In fact, this natural system is induced by the functor 
$${\mathcal R}_\sim\to\ab, \quad B\mapsto {\sf Z}(B).$$
In this way one obtains the element in $H^3({\mathcal R}_\sim,  D^{\mathcal R})$, which is denoted by $\grupo{\mathcal R}$ and is called the \emph{global Teichm\"uller class}. 
  
For a fixed ring $A$ we can consider a track subcategory of $\mathcal R$, which has just one single object $A$ and the 1-morphisms of which are automorphisms. The tracks $f\Longrightarrow g$ are the same as in $\mathcal R$. Then this track subcategory is the track category corresponding to the crossed extension $e_A$. Hence we obtain
$$h_A=Res(\grupo{\mathcal R}).$$

\section{Applications to graded extensions of (monoidal) categories}

\subsection{Recollection on Grothendieck cofibrations}

Let $P:{\mathcal E}\to {\mathcal B}$ be a functor and $B$ be an object of ${\mathcal B}$. The \emph{fibre} ${\mathcal E}_B$ of $P$ over $B$ is the subcategory of ${\mathcal E}$ consisting of all morphisms $f$ such that $P(f)=id_B$. In particular, the objects of ${\mathcal E}_B$ are such objects $E$ of ${\mathcal E}$ that $P(E)=B$. Denote by $i_B:{\mathcal E}_B \hookrightarrow {\mathcal E}$ the inclusion functor. We say that a morphism $f:E\to F$ in the category  ${\mathcal E}$ is \emph{over} $a:A\to B$ if $P(f)=a$. This of course implies that $E\in  {\mathcal E}_A$ and $F\in {\mathcal E}_B$.

A morphism $f:E\to F$ in the category  ${\mathcal E}$ is called \emph{cocartesian} over $a=P(f):A\to B$, if each morphism $g:E\to G$ in $\mathcal E$ and each decomposition of $b=P(g)=c\circ a$ in $\mathcal B$ uniquely determines a morphism $h:F\to G$ in $\mathcal E$ over $c$ with $g=h\circ f$:
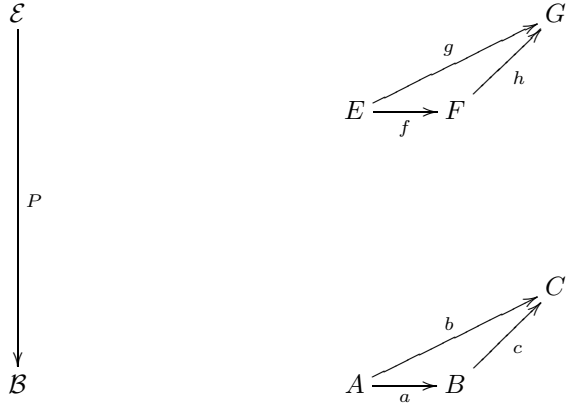
$$\xymatrix{ &&&&&&&\\
{\mathcal E}\ar[dddd]^P && &&&& G\\
&&&& E\ar[r]_f\ar[rru]^g&F\ar[ru]_h&\\
&&&&&&&\\
& && &&& C\\
{\mathcal B} &&&& A\ar[r]_a \ar[urr]^b &B\ar[ur]_c\\
&&&&&&&\\
}
$$

The functor $P:{\mathcal E}\to {\mathcal B}$ is a \emph{cofibration} if above each morphism $A=P(E)\to B$ in $\mathcal B$ there is a cocartesian morphsim $E\to F$. 

A \emph{cleavage} in a cofibration $f:E\to F$ is a choice, for each object $E$ of $\mathcal E$ and morphism $a:A=P(E)\to B$ in $\mathcal B$, of a cocartesian morphism $a_*:E\to F$ above $a$. If $P$ is equipped with a cleavage, it is said to be \emph{cloven}. The cleavage defines the functor  $a_*:{\mathcal E}_A\to {\mathcal E}_B$, which sends an object $E$ over $A$ to the codomain of the cocartesian morphism $a_*$.  In fact, 
$$A\mapsto {\mathcal E}_A,\ a\mapsto a_*$$ yields a $2$-functor ${\mathcal B}\to {\mathcal Cat}$. Conversely, having a $2$-functor   ${\mathcal B}\xto{\Psi} {\mathcal Cat}$, one can construct the category ${\mathcal E}={\mathcal B}\int \Phi$, known as the Grothendieck construction. The objects of ${\mathcal B}\int \Phi$ are pairs $(A,x)$, where $A$ is an object of $\mathcal B$ and $x$ is an object of $\Psi(A)$. A morphism $(A,x)\to (B,y)$ in ${\mathcal B}\int \Phi$ is a pair $(a,\alpha)$, where $a:A\to B$ is a morphism in $\mathcal B$, while $\alpha:\Phi(a)(x)\to y$ is a morphism in $\Psi(B)$. The composite of morphisms $(A,x)\xto{(a,\alpha)} (B,y)$ and $(B,y)\xto{(b,\beta)}$ is $(A,x)\xto{(c,\gamma)} (C,z)$, where $c=ba$, while $\gamma$ is the composite
$$c_*(x)=\Phi(ba)(x)\to \Phi(b)(\Phi(a)(x))\xto{\Phi(\alpha)(x)} \Phi(b)(y)\xto{\beta} z.$$

Consider the functor $P:{\mathcal B}\int \Phi\to {\mathcal B}$ , where
$$P(A,x)=A \quad {\rm and} \quad P(a,\alpha)=a.$$
One easily checks that the morphisms of the form $(a,id)$ are cocartesian and hence $P$ is a cofibration. Grothendieck proved that in this way one obtains a one-to -one correspondence between cofibarions over ${\mathcal B}$ (up to equivalence) and $2$-functors from ${\mathcal B}\to {\mathcal Cat}$.

\begin{Le} If ${\mathcal E}$ and ${\mathcal B}$ are groupoids and $P$ is full, then 
$P$ is a cofibration. 
\end{Le}

\begin{proof} The result follows from the fact that any isomorphism is cocartesian. 
\end{proof}

%Assume now there is given a $2$-functor ${\mathcal B}\to {\mathcal Mon.Cat}$ to the $2$-category of monoidal categories, monoidal functors and monoidal transformation. In this case the coresponding cofibration $ {\mathcal B}\int \Phi\to {\mathcal B}$ has a ${\mathcal B}$-monoidal category in the following sense.
%
%In order to formulate this notations, we consider the pull-back
%$$\xymatrix{{\mathcal E}\times_{\mathcal B} {\mathcal E}\ar[r]\ar[d] & {\mathcal E}\ar[d]^P\\
%{\mathcal E}\ar[r]_P& {\mathcal B}
%}$$ 
%Thus objects of ${\mathcal E}\times_{\mathcal B} {\mathcal E}$ are pairs $(E,F)$, where $E$ and $F$ are objects in ${\mathcal E}$ such that $P(E)=P(F)$. Moreover,  a morphism $(E,F)\to (E',F')$ is a piar of morphism $f:E\to E'$ and $g:F\to F'$ in $\mathcal E$ such that $P(f)=P(g)$.
 
%A cofibration ${\mathcal E}\xto{P} \mathcal B$ is ${\mathcal B}$-monoidal category, if there is given a bifunctor
%$$\otimes:{\mathcal E}\times_{\mathcal B} {\mathcal E}\to {\mathcal E}$$
%with properties

\subsection{The class of Cegarra-Garz\'on-Grandjean}

In \cite{cegarra}, the authors define an extension of a category $\mathcal{C}$ by a group $G$. In the case when $\mathcal C$ is the category $\underline{\Pi}$ associated to a group $\Pi$ one recovers the classical theory of group extensions.

A stable $G$-grading on a category $\mathcal{D}$ is a functor $g:\mathcal{D}\to \underline{G}$ such that for every object $A\in \mathcal{D}$ and $x\in G$, there is an isomorphism $\kappa$ in $\mathcal{D}$ with source $A$ and such that $g(\kappa) = x$. We refer to $g(\kappa)$ as the grade of $\kappa$. Then we can define the category $Ker(\mathcal{D})$ as the subcategory consisting of all morphisms of grade $1$.

A \emph{$G$-graded extension} \cite{cegarra} of a category $\mathcal{C}$ is a stably $G$-graded category whose kernel is isomorphic to $\mathcal{C}$.

In this setting, by a factor set, or $2$-cocycle on $G$ with coefficients in $\mathcal{C}$, we shall mean a $2$-functor $\underline{G} \to \emph{Cat}$, the category of small categories, that associates the category $\mathcal{C}$ to the unique object of $\underline{G}$. This is a special case of the Grothendieck construction where $\underline{G}$ is the one-object category associated to the group $G$, rather than a general (small) category.

The centre of $\mathcal{C}$, ${\sf Z}(\mathcal{C})$, is defined as the set of all natural transformations $u:id_{\mathcal{C}}\to id_{\mathcal{C}}$, where $id_{\mathcal{C}}$ is the identity functor, and ${\sf Z}(\mathcal{C})^*$ denotes the abelian group of the units in ${\sf Z}(\mathcal{C})$, that is, the abelian group of all natural isomorphisms of $id_\mathcal{C}$ with itself. The group of outer autoequivalences of $\mathcal{C}$, $\Out(\mathcal{C})$, is the set of isomorphism classes of autoequivalences of $\mathcal{C}$ with the multiplication induced by the composition of autoequivalences.

In this setting, we also have a theory of obstructions, and a collective character $\Phi:G\to \Out(\mathcal{C})$ for such an extension plays a similar role to an abstract kernel. To it, we can associate an element $k(\Phi)\in H^3_\Phi(G, {\sf Z}(\mathcal{C})^*)$, called the \emph{Teichm{\"u}ller class}, which is constructed analogously to a classic construction by Teichm{\"u}ller for a similar situation with linear algebras, and to that by Eilenberg--Mac Lane, for the study of obstructions to group extensions with non-abelian kernels considered above.

\begin{Th}
\cite{cegarra} A collective character $\Phi:G\to \Out(\mathcal{C})$ is realizable if and only if its
Teichm{\"u}ller obstruction class $k(\Phi)\in H^3_\Phi(G, {\sf Z}(\mathcal{C})^*)$ vanishes.
\end{Th}

Let us consider the track category ${\mathcal Cat}$, the objects of which are small categories, and the $1$- and $2$-morphisms are functors and natural transformations, respectively.

\begin{Le}\label{niso_abe}
If we restrict the $1$-morphisms in the track category ${\mathcal Cat}$ to equivalences and the $2$-morphisms to natural isomorphisms, we will have defined an abelian track category.
\end{Le}

\begin{proof}
We need to show that for each equivalence $f:C\to D$ in ${\mathcal Cat}$, the group $\Aut(f)$ is abelian. In order to do so, we show that for any two automorphisms $\alpha$, $\beta$ of $f$, there exists a unique isomorphism $u\in \Aut(Id_D)$ such that $\beta = f^*u+\alpha=\alpha+f^*u$. In fact, $\alpha+\beta=\alpha+(f^*u+\alpha)=(\alpha+f^*u)+\alpha=\beta+\alpha.$

For any object $Y\in D$, since $f$ is essentially surjective, we can select an object $X\in C$ and an isomorphism of $D$, $\eta_{X,Y}: f(X)\to Y$. We then define $u_Y:Y\then Y$, $Y\in D$, by $u_Y = \eta_{X,Y} -\alpha_X + \beta_X -\eta_{X,Y}$. First, we observe that this morphism does not depend on the choice of $\eta$. Indeed, for another family of isomorphisms $\mu_{X',Y}: f(X')\to Y$, since $f$ is fully faithful, there will exist a unique isomorphism $\phi_{X,X'}:X\to X'$ in $C$, such that $f(\phi_{X,X'}) = -\mu_{X',Y}+\eta_{X,Y}$. Hence, by the naturality of $\alpha$ and $\beta$, we have that $u_Y = \eta_{X,Y} -\alpha_X + \beta_X -\eta_{X,Y} = \mu_{X',Y} +f(\phi_{X,X'}) -\alpha_X+\beta_X -f(\phi_{X,X'}) - \mu_{X',Y} = \mu_{X',Y} -\alpha_{X'}+\beta_{X'} - \mu_{X',Y}$.
Next, we check that $u$ is a natural transformation. Let $\psi:Y\to Y'$ be a morphism in $D$. Since $f$ is an equivalence, there exists a morphism $\phi:X\to X'$ in $C$ such that $f(\phi_{X,X'})=-\eta_{X',Y'}+\psi_{Y,Y'}+\eta_{X,Y}$. Then the naturality of $\alpha$ and $\beta$ implies that 
\begin{align*}
\psi_{Y,Y'}+u_Y+\eta_{X,Y}&=\psi_{Y,Y'}+\eta_{X,Y} -\alpha_X + \beta_X= \eta_{X',Y'}+f(\phi_{X,X'})-\alpha_X + \beta_X \\
&=\eta_{X',Y'}-\alpha_X' + \beta_X'+f(\phi_{X,X'})=u_{Y'}+\psi_{Y,Y'}+\eta_{X,Y},
\end{align*}
and so $\psi_{Y,Y'}+u_Y=u_{Y'}+\psi_{Y,Y'}$.

For all $X\in C$, we have $u_{f(X)}=-\alpha_X + \beta_X$, implying that $\beta = \alpha+f^*u$. The equation TR 9 provides the second part of the equation.

As for the uniqueness of $u$, let us suppose that $\alpha+f^*u= \alpha+f^*v$ for $u,v\in \Aut(Id_D)$. Because $\alpha$ is an isomorphism, this implies that $f^*u=f^*v$, and so $u=v$ because $f$ is an equivalence.
\end{proof}

The abelian track category $\mathcal Cat$ defined above defines a class $\grupo{{\mathcal Cat}}$ in the third cohomology $H^3({\mathcal Cat}_\sim, D^{{\mathcal  Cat}})$.

\begin{Le}
For every small category ${\mathcal Cat}$, we have the restriction 
$$r:H^3({\mathcal Cat}_\sim, D^{{\mathcal Cat}})\to H^3(\Out({\mathcal Cat}),{\sf Z}({\mathcal Cat})^*),$$
with $k(\Phi) = \Phi^*(r(\grupo{{\mathcal Cat}}))$.
\end{Le}

\subsection{The class of Cegarra-Garz\'on-Ortega}
We would like to apply the Toda class to monoidal categories. To obtain an appropriate  abelian track category we consider the following $2$-category ${\mathcal Mcat}$. Objects of ${\mathcal Mcat}$ are small monoidal categories, $1$-morphisms are monoidal equivalences and $2$-morphisms are natural isomorphisms. It follows from Lemma \ref{niso_abe} that ${\mathcal Mcat}$ is an abelian track category. The corresponding homotopy category ${\mathcal Mcat}_\sim$ is a groupoid. As any abelian track category, it defines a class
$$\grupo{{\mathcal Mcat}}\in H^3({\mathcal Mcat}_\sim, D^{\mathcal Mcat}).$$
Here $D^{\mathcal Mcat}$ is a natural system such that for any monoidal equivalence $f:({\mathcal C},\otimes)\to {\mathcal C'},\otimes')$ the abelian group $D^{\mathcal Mcat}_{q(f)}$ is isomorphic to the following group
$${\sf Z}_f=\{\alpha:f\then f| \ \alpha \ {\rm is \ a} \ {\rm monoidal} \ {\rm isomorphism}\}.$$
Here as usual $q$ denotes the natural functor ${\mathcal Mcat}_0 \to {\mathcal Mcat}_\sim$.

For $f=\id_{\mathcal C}$, the group ${\sf Z}_{\id_{\mathcal C}}$ is known as the \emph{centre} of the monoidal category $({\mathcal C},\otimes)$,  and is denoted by ${\sf Z}({\mathcal C},\otimes)^*$, see \cite[p.631]{ortega}.

For a monoidal category $({\mathcal C}, \otimes)$ the group of automorphisms of $({\mathcal C}, \otimes)$ in ${\mathcal Mcat}_\sim$ is known as the Picard group of $({\mathcal C}, \otimes)$ \cite{ortega} and is denoted by ${\sf Pic}({\mathcal C}, \otimes)$, compare with \cite[p.633]{ortega}. Thus ${\sf Pic}({\mathcal C}, \otimes)$ is the set of isomorphic classes of monoidal autoequivalences of $({\mathcal C}, \otimes)$, with the multiplication induced by the composition of monoidal autoequivalences.  

Let $G$ be a group. A group homomorphism $\varrho\to {\sf Pic}({\mathcal C}, \otimes)$ can be seen as a functor $\underline{\varrho}:\underline{G}\to {\mathcal Mcat}_\sim$, which sends the unique object of $\underline{G}$ to $({\mathcal C},\otimes)$. Hence we obtain the class
$${\mathcal T}(\varrho):=\varrho^*(\grupo{{\mathcal Mcat}})
\in H^3(G, {\sf Z}({\mathcal C},\otimes)^*)$$
which was considered in \cite[637]{ortega}. The map given by $\varrho\mapsto {\mathcal T}(\varrho)$ is known as the \emph{Teichm\"uller obstruction map}, \cite[637]{ortega}. The name comes from the problem to lift the homomorphism $\varrho$, equivalently the functor $\underline{\varrho}:\underline{G}\to {\mathcal Mcat}_\sim$ to a $2$-functor $\underline{G}\to {\mathcal Mcat}$. Thanks to Lemma \ref{1.5.m} such a lifting exists iff ${\mathcal T}(\varrho)=0$ and if this is the case. then set of equivalence classes of such liftings is a torsor over $H^2(G, {\sf Z}({\mathcal C},\otimes)^*)$. So, we obtain the main results of \cite{ortega}, see Proposition 4.1. Theorem 5.1 and Theorem 5.2 of \cite{ortega}. 
 
Moreover, a variant of this theory is also considered in \cite{ortega}, where instead of monoidal categories , so called $k$-linear monoidal categories are considered, where $k$ is a fixed commutative ring. By considering abelian track category of all small $k$-linear monoidal categories, $k$-linear monoidal equivalences and $k$-linear monoidal natural isomorphisms, one obtains similar results for $k$-linear monoidal categories. Details left to an interested reader.

\subsection{Obstruction class of Chen-Du-Wang}

In the recent preprint \cite{chen}, the authors extended the theory of group extensions to groupoids. They, in fact, considered extensions of Lie groupoids. We will show how this theory (in the discrete case) can be obtained as a particular case of our approach. %In a forthcoming paper we will extend Baues-Wirsching cohomology for topological and differentiable categories and we will fully recover the results of \cite{chen}. 

Let $\sf K$ be a groupoid. All groupoids are assumed to be nonempty. The set of objects of ${\sf K}$ will be denoted by ${\sf K}_0$, while ${\sf K}_1$ denotes the set of morphisms of ${\sf K}$. For an arrow $\alpha\in {\sf K}_1$, we let $s_{\sf K}(\alpha)$ be the source of $\alpha$, while $t_{\sf K}(\alpha)$ denotes the target of $\alpha$.

A \emph{Chen-Du-Wang extension} (compare with \cite[Definition 3.1] {chen}) of a groupoid $\sf K$ by a groupoid $\sf A$ is the following data:

\begin{enumerate}[i)]

\item A groupid $\sf G$ and a morphism of groupoids $\varPhi=(\varPhi_0,\varPhi_1): {\sf G}\to {\sf K}$.

\item The identification ${\sf G}_i={\sf A}_i\times {\sf K}_i$, such that $\varPhi_i$ is the projection to the second factor, $i=1,2$.

\item The source map satisfies $s_{\sf G} =s_{\sf A}\times s_{\sf K}$.

\item For objects $k\in {\sf K}_0$, $a\in {\sf A}_0$, one has
$$ \id_{(a,k)}=(\id_{a},\id_{k}).$$

\item The target map satisfies the equation
$$t_{\sf G}(\alpha, \id_k)=(t_{\sf A}(\alpha), k).$$
Here $k\in {\sf K}_0$ and $\alpha$ is a morphism of $\sf A$. 

\item For any composable pair of morphisms $\alpha$ and $\beta$ of $\sf A$, one has

$$(\alpha,\id_k)\circ (\beta,\id_k)=(\alpha\beta,\id_k).$$

\end{enumerate}

If ${\sf G}$ is a Chen-Du-Wang extension of a groupoid ${\sf K}$ by ${\sf A}$, we write
$$1\to {\sf A}\to {\sf G}\xto{\varPhi} {\sf K}\to 1.$$

\begin{Le} If $1\to {\sf A}\to {\sf G}\xto{\varPhi} {\sf K}\to 1$ is a Chen-Du-Wang extension, then $\varPhi$ is a Grothendieck cofibration and for any object $k\in {\sf K}_0$, the maps given by
$$a\mapsto (a,k), \ \ \alpha\mapsto (\alpha,\id_k), \ a\in {\sf A}_0, \alpha\in {\sf A}_1$$
yield an isomorphism of groupoids
$${\sf A}\to \Phi^{-1}(k).$$
\end{Le} 

\begin{proof} According to the Lemma \ref{gr-ful=cof}, we only need to check that $\varPhi$ is full, but this follows from the condition ii). According to conditions iv) and vi) these maps define indeed a functor. Objects of $\Phi^{-1}(k)$ have the form $(a,k)$, where $a\in {\sf A}_0$. Moreover, morphisms in $\Phi^{-1}(k)$ have the form $(\alpha,\id_k)$ and hence the above functor is an isomorphism.  
\end{proof}

As in any Grothendieck cofibration, the assignment $k\mapsto \Phi^{-1}(k)$ defines a $2$-functor. Hence we immediately obtain the following fact (compare with \cite[Section 3.3]{chen}).

\begin{Co} If $1\to {\sf A}\to {\sf G}\xto{\varPhi} {\sf K}\to 1$ is a Chen-Du-Wang extension, then there is a $2$-functor from the groupoid $\sf K$ to the 2-category of groupoids, which has the same value ${\sf A}$ on each object of ${\sf K}$ and for the morphism $\kappa:k_1\to k_2$ of the groupoid $\sf K$, the induced functor $$\kappa_*:{\sf A}=\Phi^{-1}(k_1)\to \varPhi^{-1}(k_2)={\sf A}$$
	is given on objects by
	$$\kappa_*(a)=b$$
	where the $(b,k_2)$ is the target of $(\id_a,\kappa)$.
\end{Co} 

Conversely, having such a $2$-functor $\varPsi$, the Grothendieck construction ${\sf K}\int \varPsi$ gives a Chen-Du-Wang extension, which explains \cite[Theorem 4.1]{chen}.

Hence the classification of Chen-Du-Wang extensions completely reduces to the study of appropriate $2$-functors, which can be done based on properties of the Toda class of the abelian track category ${\sf SGpd}$. Objects of ${\sf SGpd}$ are small groupoids, $1$-morphisms are isomorphisms of groupoids and $2$-morphisms are natural isomorphisms of isomorphisms of groupoids. We will see that the corresponding Toda class $\grupo{{\sf SGpd}}\in H3({\sf SGpd}_\sim,D^{\sf SGpd})$ explains several results obtained in \cite{chen}. First of all, observe that for any small groupoid $\sf A$, we have
$$D^{\sf SGpd}_{id_{\sf A}}=\{\xi:\id_{\sf A}\then \id_{\sf A}\}.$$ Thus
$D^{\sf SGpd}_{id_{\sf A}}={\sf Z}({\sf A})$ is the centre of $\sf A$. The second observation is the fact that the full subcategory of  ${\sf SGpd}_\sim$ corresponding to the object $\sf A$ is the one object category corresponding to the group (in the notations \cite{chen}) $\overline{{\sf SAut(A)}}$, which is the group of isomorphism classes of automorphisms of ${\sf A}$. Hence by restricting the class $\grupo{{\sf SGpd}}$ to this subcategory one obtains the class 
$$Res_{\sf A}{\grupo{{\sf SGpd}}} \in H^3(\overline{{\sf SAut(A)}}, {\sf Z}({\sf A})).$$

Let us take a small groupoid ${\sf K}$ and let $\bar{\Lambda}$ be a functor from ${\sf K}$ to the category ${\sf SGpd}_\sim$ such that all objects of $\sf K$ map to ${\sf A}$. Such a functor is called a \emph{band} in \cite[Definition 3.4]{chen}. The problem still to be answered is under what conditions $\bar{\Lambda}$ can be lifted to a $2$-functor $\Lambda:{\sf K}\to {\sf SGpd}$. By Lemma \ref{1.5.m} this happens if and only if the class
$$\bar{\Lambda}^*(Res_{\sf A}{\grupo{{\sf SGpd}}}\in H^3({\sf K}, {\sf Z}({\sf A}))$$
is zero. Moreover, if this happens then isomorphism classes of such liftings form a torsor over $H^2({\sf K}, {\sf Z}({\sf A}))$. These results reprove \cite[Theorem 4.4, Theorem 4.5 and Theorem 5.3]{chen}.

%As we can replace the natural system $D^{\mathcal Gpd}$ with a functor ${\mathcal F}$, this becomes $H^3({\mathcal Gpd}, {\mathcal F})$. The factor set 

\end{document}